\documentclass[11pt,openany,leqno]{amsart}
\pdfoutput=1
\usepackage{amsmath,amsthm,amsfonts,amssymb,amscd,url}
\usepackage{graphics,todonotes}
\usepackage[latin1]{inputenc}
\usepackage{enumerate}
\usepackage{hyperref}
\makeindex

\def\Z{\mathbb Z}
\def\R{\mathbb R}

\def\loc{\textup{loc}}

\def\det{\textup{det}\,}
\def\exp{\textup{exp}}
\def\fix{\textup{Fix}\,}

\def\inv{^{-1}}

\def\loc{\textup{loc}}
\newcommand{\Cst}{\textup{Const}}
  \def\eps{\epsilon}    
    \def\HD{\textup{dim}_{\textup{H}}\,}
\def\mult{\kappa} \theoremstyle{plain} \newcommand{\Per}{\textup{Per}\,} \newtheorem{theo}{\bf Theorem}[section]
\newtheorem{lemm}[theo]{\bf Lemma} \newtheorem{sublemm}[theo]{\bf Sublemma} \newtheorem{prop}[theo]{\bf
  Proposition}
\newtheorem{coro}[theo]{\bf Corollary}
 \newtheorem{Fact}[theo]{\bf Fact} \newtheorem*{mthm}{\bf Main Theorem}
\theoremstyle{remark} \newtheorem{rema}[theo]{\bf Remark} 
  \newtheorem{exem}[theo]{\bf Example}
 \newtheorem{defi}[theo]{\bf Definition}

\title{On the Hausdorff dimension of Newhouse phenomena}
\date{\today}
\author{Pierre Berger}
\address{Pierre Berger\\
  Laboratoire Analyse, G\'eom\'etrie \& Applications 
  CNRS-UMR 7539 Institut Galil\'ee Universit\'e Paris 13\\
99 avenue J.B. Cl\'ement - 93430 Villetaneuse - France}
\email{{\tt pierre.berger@math.univ-paris13.fr}}
\urladdr{\href{http://www.math.sunysb.edu/~berger/}{http://www.math.sunysb.edu/~berger/}}
\author{Jacopo De Simoi}
\address{Jacopo De Simoi\\
  Department of Mathematics\\
  University of Toronto\\
  40 St George St. Toronto, ON, Canada M5S 2E4}
\email{{\tt jacopods@math.utoronto.ca}}
\urladdr{\href{http://www.math.utoronto.ca/jacopods}{http://www.math.utoronto.ca/jacopods}}

\thanks{PB is partially financed by the Balzan project of J. Palis and the Brazilian-French Network in Mathematics. JDS
  acknowledges partial NSERC support.}
\begin{document}

\begin{abstract}
  We show that at the vicinity of a generic dissipative homoclinic unfolding of a surface diffeomorphism, the Hausdorff
  dimension of the set of parameters for which the diffeomorphism admits infinitely many periodic sinks is at least
  $1/2$.
\end{abstract}
\maketitle
\section*{Introduction}%
The main goal of Palis program~\cite{Pa00} is to prove that given a generic, $d$-dimensional parameter family of
$C^r$-diffeomorphisms of a compact manifold ($r\ge 1$, $d\ge 1$), Lebesgue almost every parameter has the following
property:
\begin{quote}
  \textit{There exist finitely many invariant probabilities such that the
    union of their basins has full Lebesgue measure in the manifold.}
\end{quote}
According to the Palis program, homoclinic tangencies represent the main obstructions for a surface diffeomorphism to be
uniformly hyperbolic.  More precisely, Palis conjectured that every smooth surface diffeomorphism which is not uniformly
hyperbolic can be perturbed to one having a hyperbolic periodic point $P$ whose stable and unstable manifolds have a
quadratic tangency.  This result has been proved in the $C^1$-topology by Pujals and Sambarino (see~\cite{PS00}; see
also~\cite{Pujals-Crovisier} for a weaker $C^1$-version in higher dimension, and~\cite{MSS,Dujardinlyubitch} for a
weaker version in complex dynamics).  Homoclinic tangencies are associated to a remarkable dynamical feature discovered
by Newhouse (see~\cite{Newhouse,Ne79}) and which is nowadays referred to as the \emph{Newhouse phenomenon}: in any
neighborhood of a dissipative surface diffeomorphism exhibiting a non-degenerate homoclinic tangency, there exists a
residual set of diffeomorphisms which admit infinitely many periodic sinks.  Adapting some of Newhouse's results,
Robinson (see~\cite{Ro83}) later proved that the phenomenon takes place for a residual set of parameters in a
one-parameter family of diffeomorphisms which non-degenerately unfold an homoclinic tangency.

It is natural to ask whether this topologically significant behavior is also relevant from the point of view of
probability (or more precisely, of \emph{prevalence}).  The first attempt towards a measure theoretic understanding of
Newhouse phenomena is due to Tedeschini-Lalli and Yorke (see~\cite{TLY}): they considered a one-parameter unfolding of a
homoclinic tangency involving a linear horseshoe, and showed that the set of parameters whose corresponding
diffeomorphism admits infinitely many periodic \emph{simple sinks} (i.e. sinks obtained with the Newhouse construction)
is a null set for Lebesgue measure.  In the same setting (one-parameter unfolding of an homoclinic tangency involving a
linear horseshoe) Wang (see~\cite{Wang90}) proved that the Hausdorff dimension of the parameter set of diffeomorphisms
admitting an infinite number of periodic simple sinks is strictly positive and smaller than $1/2$.

More recently, Gorodetski and Kaloshin (see~\cite{Kaloshin}) obtained the mea\-sure-ze\-ro result in a much broader
setting: they introduced a quantitative notion of combinatorial complexity of periodic orbit visiting a neighborhood of a
homoclinic tangency, which they call \emph{cyclicity}\footnote{ In their terminology, \emph{simple sinks} which were
  considered above correspond to \emph{cyclicity one sinks}.}.  Their result shows that a \emph{prevalent} dissipative
surface diffeomorphism in a neighborhood of one exhibiting a non-degenerate homoclinic tangency has only finitely many
sinks of cyclicity which is either bounded or negligible with respect to the period of the orbit.

The techniques of~\cite{Ne79,Ro83} do not apply to conservative surface diffeomorphisms; on the other hand, a clear
analog of the Newhouse phenomenon still occurs in a vicinity of conservative diffeomorphisms exhibiting non-degenerate
homoclinic tangencies, with elliptic islands filling in for the r\^ole of sinks.  This result was finally established by
Duarte and Gonchenko--Shilnikov (see~\cite{Duarte99,Gonchenko-Shilnikov} and~\cite{Duarte08} for the one-parameter
version).  In~\cite{JMD} the second author of this article proved an analog to Tedeschini-Lalli--Yorke and Wang result
for the Standard Family of conservative diffeomorphisms in the large parameters regime: the set of (sufficiently large)
parameters for which the Standard Family admits infinitely many simple sinks has zero Lebesgue measure and its Hausdorff
dimension is not smaller than $1/4$.

In this paper we obtain a similar lower bound on the Hausdorff dimension for dissipative surface diffeomorphisms.  We
prove that the Newhouse parameter set for a generic family of sufficiently smooth diffeomorphisms (nondegenerately)
unfolding a homoclinic tangency has Hausdorff dimension not smaller than $1/2$.  It is important to stress that our
lower bound takes into account non-simple sinks (and thus does not contradict Wang's result) and moreover does not
assume linearity of the horseshoe.  The proof of our result hinges on two crucial ingredients: the first one
(Theorem~\ref{Newhouserev}) is an improved version of Newhouse construction of a wild hyperbolic set; the second one is
Lemma~\ref{MisuRen} (proved in~\cite{MisuRen}), which provides precise estimates on the length of the \emph{stability
  range} of a sink which is created by unfolding a homoclinic tangency via the Newhouse construction.  As a consequence
of these results, we obtain that the Hausdorff dimension of the simple Newhouse parameter set for strongly dissipative
H\'enon like families is close to one $1/2$ but not greater than $1/2$ (see Corollary~\ref{c_upperBound}).  We conclude
by using Palis-Takens renormalization (see Theorem~\ref{PT2}).


The paper is organized as follows: in Section~\ref{s_statement} we give a precise statement of our Main Theorem, and in
Section~\ref{s_proof} we collect the results used to give its proof; Appendix~\ref{s_proofNewhouse} provides a proof of
Theorem~\ref{Newhouserev} based on renormalization of H\'enon-like maps.

\section{Statement of the Main Theorem}\label{s_statement}
\subsection{Main definitions}
\label{ss_mainDefs} 
Let $f$ be a $C^r$-diffeomorphism of a surface $M$, $r\ge 1$.  A point $p\in M$ is a \emph{periodic sink} if $p$ is
periodic, $f^n(p)=p$, and all eigenvalues of the differential $D_pf^n$ have modulus less than $1$.  A point $p\in M$ is
a \emph{saddle periodic point} if $D_pf^n$ has one eigenvalue of modulus less than $1$ the another one of modulus
greater than $1$.  The \emph{local stable and unstable manifolds} of $p$ are respectively:
\begin{align*}
  W^s_\epsilon(p; f) &:= \{y\in M: \epsilon> d(f^n(p),f^n(y))\to 0,\; 0\le n\to +\infty\},\\
  W^u_\epsilon(p; f) &:= \{y\in M: \epsilon> d(f^n(p),f^n(y))\to 0,\; 0\ge n\to -\infty\}.
\end{align*}
By Hadamard--Perron Theorem, for sufficiently small $\epsilon$, they are embedded $C^r$-curves; on the other hand the
\emph{stable and unstable manifolds} of $p$
\begin{align*}
  W^s(p; f)&:= \bigcup_{n\le 0} f^n(W^s_\epsilon(p; f))&%
  W^u(p; f)&:= \bigcup_{n\ge 0} f^n(W^u_\epsilon(p; f))
\end{align*}
are immersed submanifolds.  We say that a saddle point $p$ has a \emph{homoclinic tangency} if these two immersed
submanifolds are tangent at a point.

A family $(f_a)_{a\in \mathbb R}$ of diffeomorphisms of a surface $M$ is of class $C^r$ if the map $ \mathbb R \times
M\ni (a,z)\mapsto f_a(z)\in M$ is of class $C^r$.  It is well known that, if $f_0$ has a saddle fixed point $\Omega_0$,
then this point persists as a saddle fixed point $\Omega_a$ of $f_a$ for $a$ small.  Hence there exists a $C^r$-chart
$\phi_a$ of a neighborhood $D$ of $\Omega_a$ which maps $\Omega_a$ to $0$, $W^s_\epsilon(\Omega_a; f_a)$ onto
$\{0\}\times (-1,1)$ and $W^u_\epsilon(\Omega_a; f_a)$ onto $(-1,1)\times \{0\}$.  By~\cite{HPS}, the following map can
be chosen to be $C^r$:
\[\phi\colon (-\eta,\eta)\times D\ni  (a, z) \mapsto (a,\phi_a(z))\in  (-\eta, \eta) \times \R^2.\]
We say that the family $(f_a)_a$ \emph{nondegenerately unfolds a homoclinic tangency at $a_0$ of the periodic point $\Omega$}, if
there exist $P=(p,0)\in D$ sent by $f^N_{a_0}$ to a point $Q=(0,q)\in D$, and a neighborhood $D_P\ni P$, such that, for
every $a$ sufficiently small, $f_a^ND_P\subset D$, and $f_a^N|D_P$ has the form:
\[ P+(x, y)\in D_P\mapsto Q+(\xi x^2+a +\gamma \cdot y, \zeta \cdot x)+ E_a(x, y)\in D\] where $\zeta\in \R$,
$\xi,\gamma$ are non-zero constants (all independent of $a$) and $E_a=(E_a^1,E_a^2)\in C^{r}(\R\times \R^{2}, \R^2)$
satisfies at $(x,y)=0$ and $a=0$:
\begin{equation*}
  \left\{\begin{array}{l}
      E_a^{1}=\partial_x E_a^{1}=\partial_y E_a^{1}=\partial_a E_a^{1}=\partial_{xx} E_a^{1}=0\\
      E_a^{2}=\partial_x E_a^{2}=0\end{array}
  \right.
\end{equation*}
Note that non-degenerate homoclinic unfolding tangency are stable by $C^2$ perturbations, and open and dense among
families having a homoclinic tangency at some parameter.  Let us illustrate this fact on a crucial class of maps of
$\R^2$.
\begin{defi}[H\'enon like maps]
  A H\'enon map is a map of $\R^2$ of the following form, with $a,b\in \R$:
  \[
  h_{ab}\colon \R^2\ni (x, y)\mapsto (x^2+a+y,b x)\in \R^2.
  \]
  A \emph{H\'enon $C^r$-like map} is a map $f$ of the form
  \[
  f_a\colon \R^2\ni (x, y)\mapsto (x^2+a+y,b x)+(A_a(x,y),b\cdot B_a(x,y))\in \R^2.
  \]
  where $A_a$ and $B_a$ are $C^r$-small maps.  The map $f_a$ is \emph{H\'enon $C^r$-$\delta$-like} if both $A_a$ and
  $B_a$ have $C^r$-norm at most $\delta$.

  If $A_a$ and $B_a$ depend on the parameter $a\in (-3,1)$ so that $(x,y,a)\mapsto A_a(x,y)$ and $(x,y,a)\mapsto
  B_a(x,y)$ are of class $C^r$ with norms both smaller than $\delta$, then the family $(f_a)_a$ is \emph{H\' enon
    $C^r$-$\delta$-like}.
  \end{defi}

  %
%
%
%
%
\begin{exem}[Unfolding of homoclinic tangencies in H\'enon-like families]\label{HLMAP}
  The family $(h_{a0})_a$ nondegenerately unfolds a homoclinic tangency at $a_0=-2$ of the fixed point $\Omega$, with
  $\Omega_{a_0}:=(2,0)$.  Hence for $b$ small and $\delta$ small, any H\'enon $C^2$-$\delta$-like family $(f_{a})_a$
  nondegenerately unfolds a homoclinic tangency at a certain $a\approx a_0$ of the hyperbolic continuation of $\Omega$.
\end{exem}

\subsection{Main results}
\label{ss_mainResults} 
Here is the main result of this note:
\begin{mthm}
  Let $r$ be sufficiently large (e.g. $r\ge 24$) and let $(f_\mu)_{\mu\in \mathbb R}$ be a generic family of $C^r$
  surface diffeomorphisms (nondegenerately) unfolding a homoclinic tangency at $\mu_0$ of a periodic point $\Omega$.  If
  $|\det(D_{\Omega_{\mu_0}}f_{\mu_0})|<1$, then the following set has Hausdorff dimension at least $1/2$:
  \[
  \mathcal N:= \{\mu\in \mathbb R:\; f_\mu\text{ has infinitely many sinks}\}.
  \]
\end{mthm}
\begin{rema}%
  In the Main Theorem, we need the family to be sufficiently smooth and to satisfy some genericity conditions to satisfy
  the hypotheses of a version of Sternberg linearization theorem due to Sell (see~\cite{Sell}).  Even if it seems to be
  possible to weaken the assumptions (and prove the result for e.g.\ $C^3$-families not necessarily generic), we do not
  pursue this task here.
\end{rema}
To show the main Theorem, we will prove the following result:
\begin{prop}\label{theoHen}
  For every H\'enon $C^4$-like family $(f_a)_a$, with small determinant, the following set has Hausdorff dimension at
  least $1/2$:
  \[
  \mathcal N:= \{a\in \mathbb R:\; f_a\text{ has infinitely many sinks}\}.
  \]
\end{prop}
Proposition~\ref{theoHen} implies the Main Theorem, by the following result; for
the proof see {(\cite[Theorem 1.4 and Remark 1.5]{MisuRen})}
\begin{theo}[Palis-Takens revisited]\label{PT2}
  For any $\delta>0$ and any one-parameter family $(f_\mu)_\mu$ as in the Main Theorem, there exists, for each positive
  integer $n$, a reparametrization $\mu= M_n(\tilde\mu)$ of the $\mu$ variable and $\tilde \mu$-dependent coordinate
  transformations $(\tilde x,\tilde y)\mapsto (x,y)= \Psi_{n,\tilde \mu} (\tilde x, \tilde y)$ such that:
  \begin{itemize}
  \item for each compact set $K$ in the $(\tilde x,\tilde y,\tilde \mu)$-space, the images of $K$ under the maps
    \[(\tilde x,\tilde y, \tilde \mu)\mapsto (x,y,\mu)= ( \Psi_{n,\tilde \mu} (\tilde x, \tilde y),M_n(\tilde\mu))\]
    converge, for $n\to \infty$, in the $(x,y,\mu)$ space to $(P,\mu_0)$;
  \item the domains of the maps
    \[(\tilde x,\tilde y, \tilde \mu)\mapsto \mathcal R f_{\tilde \mu}(\tilde x, \tilde y):=\Psi_{n,\tilde
      \mu}^{-1}\circ ((f|D)^n\circ (f^N|D_P))_{M_n(\tilde \mu)} \circ \Psi_{n,\tilde \mu}(\tilde x,\tilde y)\] converge,
    for $n\to \infty$, to all $\R^3$;
  \item when $n$ is sufficiently large, the family $(\mathcal R f_{\tilde \mu})_{\tilde \mu}$ is H\'enon $C^4$-$\delta$-like
    and with determinant smaller than $\delta$.
  \end{itemize}
\end{theo}
\section{Proof of Proposition~\ref{theoHen}}\label{s_proof}
The proof is based on the Newhouse construction of a wild hyperbolic set and fine estimates in parameter space on the
distribution of parameters with many attracting periodic points.
\subsection{Wild hyperbolic sets}
We recall that an invariant compact set $K\subset M$ is \emph{hyperbolic} for a diffeomorphism $f$ of $M$ if the
restriction of the tangent bundle $TM$ to $K$ splits into two $Df$-invariant directions $E^s$ and $E^u$ which are
contracted, respectively, by $Df$ and $Df^{-1}$.

The set $K$ is \emph{basic} if, moreover, it is transitive and the closure of the set of periodic points of $f$ contains
$K$.  This enables us to define, for $\epsilon>0$, as we did for periodic points, the $\epsilon$-local stable and
unstable manifolds $W^s_\epsilon (z)$ and $W^u_\epsilon (z)$ of points $z\in K$; let $W^s_\epsilon (K)=\bigcup_{z\in
  K} W^s_\epsilon(z)$ and $W^u_\epsilon (K)=\bigcup_{z\in K} W^u_\epsilon(z)$.


For every $f'$ $C^1$-close to $f$, there exists (see~\cite{Yoyointro}), a (continuous) embedding $i_{f'}$
\[
i_{f'}\colon K\hookrightarrow \R^2,
\]
so that $i_{f}$ is the canonical injection and $f'\circ i_{f'}= i_{f'}\circ f|K$.  The map $f'\mapsto i_{f'}$ is smooth
from a $C^1$-neighborhood of $f$ into the space $C^0(K,\R^2)$.  Let us denote for $f'$ $C^1$-close to $f$ and
$\Omega\in K$:
\begin{align*}
  K(f')&:= i_{f'}(K)&
  \Omega({f'})&:=i_{f'}(\Omega).
\end{align*}

$K(f')$ (resp. $\Omega({f'})$) is called the \emph{hyperbolic continuation} of $K$ (resp. $\Omega$).  In particular
$K(f)=K$.
\begin{defi}
  We say that a basic set $K$ is \emph{wild}, if there exists $L>0$ such that, for every $C^2$-perturbation $f'$ of $f$,
  there exists $\Omega_{f'}\in K$ such that $W^s_L(\Omega_{f'}(f'), f')$ and $W^u_L(\Omega_{f'}(f'), f')$ have a
  quadratic tangency.
\end{defi}
\begin{rema}
  Even if $W^s_L(\Omega_{f}(f), f)$ and $W^u_L(\Omega_{f}(f), f)$ have a quadratic tangency, there exists $f'$
  arbitrarily close to $f$ so that $W^s_L(\Omega_{f}(f'), f')$ and $W^u_L(\Omega_{f}(f'), f')$ are not tangent.  Hence
  $\Omega_{f}(f')$ is in general not equal to $\Omega_{f'}(f')$.  Since $K$ is totally disconnected, the map $f'\mapsto
  \Omega_{f'}\in K$ is not continuous; neither is the map $f'\mapsto \Omega_{f'}(f')\in K(f')$.
\end{rema}

By Example~\ref{HLMAP}, the family of maps $(f_a)_{a}$ nondegenerately unfolds a homoclinic tangency of a fixed point
$\Omega$ with small determinant.  By~\cite[Theorem 3]{Ne79}, there exists $a_0\approx -2$ such that $f_{a_0}$ leaves
invariant a wild basic set $\tilde K$.  Actually, we shall give a modern proof of this result in order to bound the
expansion of the unstable direction uniformly on the determinant of the H\'enon-like map.

\emph{In the case of family of diffeomorphisms $(f_a)_a$, to make the notation less cumbersome, $K(a)$ and $\Omega(a)$
  stand for the hyperbolic continuations $K(f_a)$ and $\Omega(f_a)$ respectively.}

\begin{theo}[Newhouse Theorem revisited]\label{Newhouserev}
  There exists $\Lambda>1$ such that for any H\'enon $C^4$-like family $(f_a)_a$ with small determinant, there exist
  $a_0\approx -2$, $L>0$, an open parameter interval $I\ni a_0$ and a basic set $\tilde K(a_0)$, such that for any
  $a_1\in I$, the following properties hold:
  \begin{enumerate}[(i)]
  \item there exists a Riemannian metric for which the expansion of $\tilde K(a_1)$ in the unstable direction is at least
    $\Lambda$;
  \item there exists $\Omega_{a_1}(a_0)\in \tilde K(a_0)$ close to $\Omega(a_0)$ so that $W^s_L(\Omega_{a_1}(a), f_{a})$
    is tangent to $W^u_L(\Omega_{a_1}(a), f_{a})$ at $a=a_1$, via a quadratic tangency, that $(f_a)_a$ nondegenerately
    unfolds.
  \end{enumerate}
  In particular, $\tilde K(a_0)$ is wild.
\end{theo}
\begin{rema}
  Item (i) in the above theorem is not proved in the versions of the Newhouse Theorem currently available in the
  literature.  It will be important that $\Lambda>1$ does not depend on the (possibly very small) determinant of $f_a$.
\end{rema}
In Appendix~\ref{s_proofNewhouse} we give a proof of Theorem~\ref{Newhouserev} based not only on Newhouse thickness but
also on renormalization techniques.
\subsection{Parameters with many attracting periodic points}
In order to get refined estimates on the distribution of parameters with many attracting periodic points, we will use the
following:
\begin{lemm}[{\cite[Lemma 3.2]{MisuRen}}]\label{MisuRen}
Let $(f_a)_a$ be a H\'enon $C^2$-like family, of class $C^4$, such that a non-degenerate homoclinic unfolding holds with a periodic point $\Omega$  in a hyperbolic set $K$ at $a=a_0$. 

Then there exist $C>0$, $M\in \mathbb N$ and $N\ge 0$ such that for every periodic point $\Omega'$ of period $p'$ close to $\Omega$, a non-degenerate homoclinic unfolding holds with $\Omega'$ at $a_0'$ close to $a_0$. 

Moreover if $p'\ge M$, by denoting by $\sigma$ the unstable eigenvalue of $D_{\Omega'} f_{a_0'}^{p'}$, there exists a parameter interval $I_{\Omega'}$  such that:
\begin{itemize}
\item for every  $a\in I_{\Omega'}$, the map $f_a$ has an attracting cycle of minimal period $p'+N$,
\item the length of $I_{\Omega'}$ is in $[\sigma^{-2}/C,C \sigma^{-2}]$,
\item the distance between $I_{\Omega'}$  and  $a_0'$ is in $[\sigma^{-1}/C, C \sigma^{-1}]$.
\end{itemize}
%
%
\end{lemm}

The above lemma provides, in our general setting, the parameter space estimate which has been essential in most previous
results; for instance, in~\cite{TLY,Wang90} this estimate follows from the linearity of the horseshoe; in our case it is
obtained via renormalization techniques.  We can employ it, \emph{en passant}, to recover a version of
Tedeschini-Lalli--Yorke and Wang results applicable to our broader setting.
\begin{coro}\label{c_upperBound}
  Given a horseshoe $K$ for a H\'enon $C^4$-like family, and a non-degenerate unfolding of a homoclinic tangency, the set $\mathcal N_1(K)$ of
  parameters with infinitely many attracting simple\footnote{ For our purposes we can define a cycle to be \emph{simple}
    (with respect to a given homoclinic tangency) if it appears as a consequence of the Newhouse construction, of which
    Lemma~\ref{MisuRen} is a refinement.  In particular, there exist a neighborhood of the tangency point which contains
    only one element of every simple cycle.} sinks has Hausdorff dimension at most $1/2$.  In particular, it is a null set
  for Lebesgue measure.
\end{coro}
\begin{proof}
  For $\epsilon>0$ small and $i\in\mathbb N$, define the intervals $I_i:= [i\epsilon, (i+1)\epsilon)$.  Let $\Per_i$ be
  the set of periodic orbits of $K$ whose mean expansion $\sigma\in I_i$ and let $K_i$ be the closure of $\Per_i$;
  clearly $K_i$ is a (hyperbolic) basic set.  We claim that $\mathcal N_1(K)$ is the union $\mathcal N_1(K)=\bigcup_i
  \mathcal N_1(K_i)$: in fact every $f\in \mathcal N_1(K)$ has infinitely many simple sinks.  Each such sink is 1-1
  associated with an integer $n$ and a fixed point of $f^n$.  By compactness, only finitely many $K_i$ are nonempty,
  thus for every $ f\in \mathcal N_1(K)$, there exists $i$ so that there are infinitely many sinks associated to
  periodic saddles in $K_i$, which implies that $f\in\mathcal N_1(K_i)$.

  Let $h_i$ be the topological entropy of $K_i$, by Ruelle inequality it holds 
	%
  \[
  h_i\le \sup \log \sigma\le \log ((i+1)\epsilon).
  \]
  On the other hand, the number of points in $\Per_i$ fixed by $f^n$ is at most $\exp(n h_i)\le ((i+1)\epsilon)^n$.
  Given $p\in\fix f^n|K_i$ of period $n$, Lemma~\ref{MisuRen} implies that the attracting cycle associated to it exists
  for an interval of parameters $I_p$ of length $\Cst\, \sigma^{-2n}\le \Cst (i\epsilon)^{-2n}$.

  The family $(I_p)_{p\in\fix f^n, n\ge N}$ is a covering of $\mathcal N_1(K_i)$ for every $N$. We notice that
  \[
  \sum_{p\in\fix f^n, n\ge N} |I_p|^s \le \Cst \sum_{n\ge N}
  (i\epsilon)^{-2ns} ((i+1)\epsilon)^n.
  \]
  The above series converges if $(i\epsilon)^{-2s}(i+1)\epsilon<1$, thus the Hausdorff dimension of $\mathcal N_1(K_i)$
  is at most $s$.  The above condition on $s$ can be made arbitrarily close to $s=1/2$ by taking $\epsilon$ small.
\end{proof}
Let us go back to the proof of our main result.  We recall that, by Theorem~\ref{Newhouserev}, for every $a_1\in I\ni
a_0$, the map $f_{a_1}$ has a wild basic set $\tilde K(a_1)$, which is the hyperbolic continuation of $\tilde K(a_0)$.
Also, there exists $L$, such that for every $a_1\in I$, there exists $\Omega_{a_1}({a_1})\in \tilde K({a_1})$ close to
$\Omega({a_1})$ so that $W^s_L(\Omega_{a_1}({a_1}), f_{a_1})$ is tangent to $W^u_L(\Omega_{a_1}({a_1}), f_{a_1})$, via a
quadratic tangency, unfolded nondegenerately by $(f_a)_a$.

Since periodic points are dense in $\tilde K$, there exists, for any $a_1$, a periodic point $\Omega'_{a_1}(a_1)$ nearby
$\Omega_{a_1}(a_1)$ and a parameter $a_1'$ nearby $a_1$, so that $W^s_L(\Omega'_{a_1}(a_1'), f_{a_1'})$ and
$W^u_L(\Omega'_{a_1}(a_1'), f_{a_1'})$ have a quadratic tangency, unfolded nondegenerately by $(f_a)_a$.

We will bound the distance between $a_1'$ and $a_1$ as a function of the unstable eigenvalue of
$D_{\Omega'_{a_1}(a'_1)} f_{a_1'}^{p}$, with $p$ being the period of $\Omega'_{a_1}$, in view of Theorem
\ref{Newhouserev}.  Since the unfolding is nondegenerate, the distance between $a_1'$ and $a_1$ is of the same order as
the distance between $\Omega_{a_1}(a_1)$ and $\Omega'_{a_1}(a_1)$.  Furthermore, by a standard computation, there exists
$C>0$, such that for all such $a_1$, $a_1'$, $\Omega_{a_1}$, $\Omega'_{a_1}$, and $n\ge 0$ we have:
\begin{align}\label{e_estimatea1a1'}
  \frac{d(a_1,a_1')}C < d(\Omega_{a_1}(a_1),\Omega'_{a_1}(a_1))< C d(a_1,a'_1)\\
  \frac nC<\frac{\partial_a \|D_{\Omega'_{a_1}} f_{a}^{n}\|}{\|D_{\Omega'_{a_1}} f_{a}^{n}\|}<Cn,\quad \forall a\in I.\notag
\end{align}

By the construction of a Markov partition by Bowen~\cite{Bo70}, there exist $N\ge 1$, a matrix $A\in M_n(\{0,1\})$ such that
$\tilde K$ is homeomorphic to $\Sigma_A: = \{(j_i)_i\in \{1,\dots , N\}^\Z: \; A_{j_ij_{i+1}}=1,\; \forall i\in \Z\}$,
via a homeomorphism $h$ which conjugates $f|\tilde K$ to the shift $\sigma\colon \Sigma_A\ni (j_i)_i \mapsto
(j_{i+1})_i\in \Sigma_A$.

Moreover, given any $\epsilon>0$, we can find $h$, $N$ and $A$ so that, for every $j\in\{1,\dots, N\}$, the set
$R_j\subset \tilde K$ of points $z$ such that the $0$-coordinate of $h(z)$ is equal to $j$ satisfies that the diameter
of $R_j$ is less than $\epsilon$.  We notice that $R_j$ is a clopen set of $\tilde K$.

An \emph{admissible chain} from $i$ to $j$ of length $m$ is a sequence $i=j_1,\dots,j_{m}=j$, such that $A_{j_kj_{k+1}}=1$
for every $k$.
\begin{Fact}\label{fait}
  By transitivity of $\tilde K$, there exists $M_T\ge 0$ such that for all $i,j\in\{1,\cdots,N\}$, there exists an
  admissible chain from $i$ to $j$ of length $m\le M_T$.
\end{Fact}
For $n<m\in \Z$,  an admissible chain  $\underline j := (j_i)_{i=n}^m\in \{1,\dots , N\}^{m-n+1}$, define the sets:
\[R_{\underline j}:= \bigcap_{i=n}^m f^i(R_{j_i})\] Let $z\in \tilde K$ and set for $n\ge 0$ (we omit the obvious
dependence on $a$):
 \begin{align*}
   \lambda_n(z)&:= \|D_{z}f^n|E^s\|&
   \sigma_{n}(z)&:= \|D_{z}f^n|E^u\|,\\
   \lambda_{-n}(z)&:= \|D_{f^{-n}(z)}f^{n}|E^s\|&
   \sigma_{-n}(z)&:= \|D_{f^{-n}(z)}f^n|E^u\|.
\end{align*}
A classical argument of binding proves the following
\begin{lemm}
  There exists $C>0$, such that for all $m\le 0\le n$, every admissible chain $\underline j := (j_i)_{i=m}^n$, and all
  $z,z'\in R_{\underline j}$, we have
  \[\lambda_{n}(z)\le C \lambda_{n}(z') \quad \text{and}\quad \sigma_{n}(z)\le C \sigma_{n}(z').\]
  \[\lambda_{m}(z)\le C \lambda_{m}(z') \quad \text{and}\quad \sigma_{m}(z)\le C \sigma_{m}(z').\]
\end{lemm}
\begin{coro}\label{diametre}
  There exists $C>0$, such that such that for all $m\le 0\le n$, any admissible chain $\underline j := (j_i)_{i=n}^m$,
  the diameter of $R_{\underline j}$ is less than $C\max(\sigma_n^{-1}(z), \lambda_{m}(z))$ for any $z\in R_{\underline
    j}$.
\end{coro}
\begin{prop}
  There exists $C>0$, such that, for any $a_1\in I$ and $n\ge 0$, with $n'=n'(n)$ such that
  \[
  \lambda_{-n'-1}(\Omega_{a_1}(a_1)) \le \sigma_n^{-1}(\Omega_{a_1}(a_1))\le \lambda_{-n'}(\Omega_{a_1}(a_1)),
  \]
  let $\underline j := (j_i)_{i=-n'}^n\in \{1,\dots , N\}^{n+n'+1}$, such that $\Omega_{a_1}$ belongs to $R_{\underline
    j}$.  Then at $a=a_1$ and $z=\Omega_{a_1}$
  \[
  (\sigma_n \sigma_{-n'} )\le C(\sigma_n)^{1-C/\log b}.
  \]
\end{prop}
\begin{proof}
  For ease of exposition, although with abuse of notation, let us denote $\lambda_{-n'}(\Omega_{a_1}(a_1))$ by
  $\lambda_{-n'}$, $\sigma_{-n'}(\Omega_{a_1}(a_1))$ by $\sigma_{-n'}$, and so on.  However we keep in mind the
  dependence in $a_1$ and $z$.

  Since the determinant of $f$ is smaller than $b$, there exists $C>0$ such that for all $a_1\in I$ and $n\in \Z$:
  \begin{equation*}
    \sigma_{n}\lambda_{n}\le Cb^{|n|}\; .
  \end{equation*}
  Moreover, since the norm of $Df$ is bounded by $5$ on the complement of the basin of infinity, and by $(i)$ of
  Theorem~\ref{Newhouserev}, there exists $C>0$ such that for all $a_1\in I$ and $n\in \Z$:
  \begin{equation*}
    5^{|n|}\ge \sigma_{n}\ge C\Lambda^{|n|}.
  \end{equation*}
  Hence there exists $C>0$ such that for all $a_1\in I$ and $n\ge 0$:
  \begin{equation*}
    \lambda_{-n'}\le C(b/\Lambda)^{n'}\le C b^{n'}.
  \end{equation*}
  But $ 5^{-n}\le \sigma_{n}^{-1}\le \lambda_{-n'}$.
  Thus
  \begin{equation*}
    5^{-n}\le Cb^{n'}\Rightarrow {n'}\le \frac{n\log 5+\log C}{-\log b}\; .
  \end{equation*}
  Consequently:
  \begin{equation*}
    \sigma_{n'} \le C 5^{\frac{\log 5}{-\log b}n} \le
    C \Lambda^{-\frac{(\log 5)^2}{\log \Lambda \log b }n}\le C\sigma_n^{-\frac{(\log 5)^2}{\log \Lambda \log b }}. \qedhere
  \end{equation*}
\end{proof}
By Corollary~\ref{diametre}, if $\Omega_{a_1}$ belongs to $R_{\underline j}$, with $\underline j(n) := (j_i)_{i=-n'}^n$,
then the diameter of $R_{\underline j(n)}$ is less than $C\sigma_n^{-1}$.

By Fact~\ref{fait}, there exists a periodic point $\Omega_{a_1}^{(n)}$ in $R_{\underline j(n)}$ of period $p(n)\in
[n+n', n+n'+M_T]$.  By~\eqref{e_estimatea1a1'} and Lemma~\ref{MisuRen} we know that $f_a$ will nondegenerately unfold a
homoclinic tangency for $\Omega_{a_1}^{(n)}(a)$ at $a=a_1^{(n)}$ that is $C\sigma_n^{-1}$-close to $a_1$. Indeed:
\[d(a_1,a_1^{(n)})\le C d(\Omega_{a_1}(a_1),\Omega_{a_1}^{(n)} (a_1))\le \sup_{a\in [a_1,a_1^{(n)}]}\frac C{ \sigma_{-n} (a)}\]
\[\le \frac C{\sigma_{-n}(a_1)}+C  d(a_1,a_1^{(n)}) \sup_{[a_1,a_1^{(n)}]} \partial_a \frac1{\sigma_{-n}(a)}=\frac C{\sigma_{-n}(a_1)} +o(d(a_1,a_1^{(n)}))\]

 Moreover,
again by Lemma~\ref{MisuRen}, if $n$ is large enough (e.g. $n\geq M$ certainly suffices), there exists an interval of
parameters $I_n$, at a distance dominated by $\sigma_n^{-1}$ of $a_1^{(n)}$ (and consequently of $a_1$), which is of
length proportional to $(\sigma_{-n'}\sigma_n)^{-2}\ge C(\sigma_n)^{-2+\frac{C}{\log(b)}}$, such that for every $a\in
I_n$, the map $f_a$ has an attracting periodic point of minimal period $p(n)+N$.  Let us summarize the content of this
subsection in the following
\begin{prop}\label{p_construction}
  There exist $\Lambda'>\Lambda>1$, $C,D>0$, an open interval $I$
  of parameters $a$, such that for every $a_1\in I$, there exist a sequence of intervals $(I_k(a_1))_k$ and a sequence
  of positive integers $(m_k)_k$ such that:
  \begin{itemize}
  \item $I_k(a_1)$ is at a distance at most $D|I_k|^{1/2 -C/\log b}$ of $a_1$,
  \item there exists $\bar M>0$ so that $m_{k+\bar M}>m_j$ for any $0<j<k$;
  \item $f_a$ has an attracting periodic point $p_k(a)$ of period $m_k$, for every $a\in I_k(a_1)$;
  \item $\Lambda |I_{k+1}(a_1)|< |I_{k}(a_1)|<\Lambda' |I_{k+1}(a_1)|$.
  \end{itemize}
\end{prop}
\subsection{Lower bound on the Hausdorff dimension of $\mathcal N$}
\newcommand{\base}{\Lambda}
\newcommand{\basex}{m}
\newcommand{\cant}{\mathcal I}
\newcommand{\indset}{J}
\newcommand{\costa}{L}
Proposition~\ref{p_construction} enables us to build subsets of $\mathcal N$ which are particularly well-suited to
obtain bounds on the Hausdorff dimension.  Let $n_l\to\infty$ be a sequence of natural numbers which diverges
super-exponentially fast; to fix ideas, take $n_l=2^{C\cdot2^l}$ where $C> 1$ is a large constant.  We will now
construct a Cantor set by the following inductive procedure.  Let $\cant_0=I$ be the interval given by
Proposition~\ref{p_construction}.  Then, for $l>0$, assume by induction that we defined sets
$\cant_0\supset\cdots\supset\cant_{l-1}$ satisfying the following properties:
\begin{itemize}
\item for any $a\in\cant_{l-1}$, $f_a$ has at least $l-1$ attracting periodic points of period at most $m_{n_{l-1}+\bar M}$;
\item for $0\leq k\leq l-1$, $\cant_k$ is a finite disjoint union $\cant_k=%
  \bigsqcup_{j}
  \cant_k^j$ where $\cant_k^j$ are intervals such that
\[
C'^{-1}\base^{- n_k}\leq|\cant_k^j|\leq C'\base^{- n_k}\;, \quad C':= \sqrt {\Lambda'};
\]
moreover, the intervals $\cant_{k}^j$ are separated by gaps of at least $\eps_k=D'\base^{-\alpha n_{k}}$, for some
$D'>0$ and $\alpha=(1/2-C/\log b)$;
\item for $0< k\leq l-1$, each interval $\cant_{k-1}^j$ contains a number $\mult_{k}^j$ of components of $\cant_{k}$ so
  that, for some $C''>0$
\[
\mult_{k}^j\geq%
C''\base^{- n_{k-1}+\alpha n_k}.
\]
\end{itemize}
We then define the set $\cant_{l}$ as follows: partition each interval $\cant_{l-1}^j$ in sub-intervals of length
ranging in $(2D'\base^{-\alpha n_l},4D'\base^{-\alpha n_l})$, with $D'$ fixed to be specified later.  For any parameter
$a$ that is an internal\footnote{ That is, we do not consider endpoints of intervals $\cant_{l-1}^j$} endpoint of this
subdivision, we apply Proposition~\ref{p_construction}; we choose $k$ (depending on $a$) so that
\[
C'^{-1}\base^{- n_{l}}\leq|I_k(a)|\leq C'\base^{- n_{l}};
\]
  By design, for all $a'\in I_k(a)$, $f_{a'}$ admits an attracting periodic orbit of period
$m_k$; moreover, our construction ensures that $\rho n_l\leq k+k_0\leq n_l$ with $\rho = \log \Lambda/\log \Lambda'$ and $k_0$ a constant independent of $k,l$; if we
choose $C$ in the definition of $n_l$ to be large enough, we can guarantee that $m_k> m_{n_{l-1}+\bar M}$, which
ensures that the periodic orbit we found at this step is distinct from the ones obtained at any of the previous steps.

Moreover, we know that the distance of $I_k(a)$ from $a$ is at most
\[
DC'^{\alpha}\base^{-\alpha n_{l}};
\]
therefore, if we let $D'=4DC'^\alpha$, such an interval will be separated from the one which
could be obtained applying this construction to any other internal endpoint of our subdivision by at least
$D'\base^{-\alpha n_{l}}$.  We repeat this construction for all internal endpoints and we define $\cant_l$ to be the disjoint
union of all such intervals.  Observe that, by definition, each interval $\cant_{l-1}^j$ contains a number $\mult_{l}^j$
of intervals of $\cant_{l}$ which can be bounded as follows:
\[
\mult_l^j\geq C'' \base^{- n_{l-1}+\alpha n_l}
\]
with $C''=(6C'D')^{-1}$.  This concludes the proof of our induction step.

We can choose $C$ sufficiently large in the definition of the sequence $n_l$ to ensure that $\mult_l^j\geq 2$ for any
$l$ and $j$.  Let $\cant=\bigcap_{l=0}^{\infty}\cant_l$; then by construction, for any $a\in\cant$, $f_a$ admits
infinitely many sinks, i.e.\ $\cant\subset\mathcal N$.  In order to obtain a lower bound on the Hausdorff dimension of
$\cant$, and thus on the Hausdorff dimension of $\mathcal N$, we use the following result (we refer the reader
to~\cite[Example 4.6]{Fa03} for its proof).
\begin{lemm}\label{l_falconer}
  Let $\cant$ be a Cantor set constructed as above so that for all $l$, every interval $\cant_{l-1}^j$ contains at least $\mult_l\geq 2$
  intervals $\cant_{l}^{j'}$, which are separated by gaps of at least $\eps_l$; assume furthermore that
  $\eps_{l}<\eps_{l-1}$.  Then
  \begin{equation}
    \HD \cant\geq \liminf_{l\to\infty} \frac{\log(\mult_1\cdots\mult_{l-1})}{-\log(\mult_l\eps_l)}.\label{e_Falconer}
  \end{equation}
\end{lemm}
By our construction, the hypotheses of the above lemma hold with $\eps_l=D'\base^{-\alpha n_{l}}$ and $\mult_l=C''
\base^{- n_{l-1}+\alpha n_l}$.  Let us then compute the $\liminf$ on the right hand side of~\eqref{e_Falconer}; for any $l>0$:
\begin{align*}
  \log(\mult_1\cdots\mult_{l-1})&=(l-1)\log C'' +\\&\phantom=+ \left({\alpha n_{l-1}+(\alpha-1)n_{l-2}+\cdots+(\alpha-1)n_{1}-
      n_0}\right)\log\base;\\
  -\log(\mult_l\eps_l)&= -\log(C''D')+ n_{l-1}\log\base.
\end{align*}
Given the super-exponential growth of $n_l$ we obtain:
\[
\liminf_{l\to\infty} \frac{\log(\mult_1\cdots\mult_{l-1})}{-\log(\mult_l\eps_l)}=\alpha=\frac12 - C/\log b.
\]
By Lemma~\ref{l_falconer} we can thus conclude that
\[
\HD\mathcal N\geq \HD \cant \geq \frac12 - C/\log b,
\]
and, since $b$ can be chosen to be arbitrarily small, Proposition~\ref{theoHen} (and thus the Main Theorem) is proved.\qed

\section{Conclusions}

In this paper we obtained a lower bound on the Hausdorff dimension of the Newhouse parameter set for a family of smooth
dissipative surface diffeomorphisms which nondegenerately unfolds a homoclinic tangency.  In order to obtain our lower
bound we take into account a specially designed class of non-simple sinks.  It is natural to wonder about the optimal of
our result: since we obtain this bound considering only a very special class of non-simple sinks, one would expect that
the Hausdorff dimension of the Newhouse parameter set could indeed be larger than $1/2$.

In the same spirit, we should mention two works the Hausdorff dimension of related pathological phenomena:\cite{L98}
and~\cite{AM02}. The first shows that the Hausdorff dimension of infinitely renormalizable parameters is at least $1/2$
(by focusing only on Misiurewicz renormalization), the second shows that this set has Hausdorff dimension less than $1$.
The latter shows also that the Hausdorff dimension of various other pathological phenomena (inexistence of physical
measures, physical measures supported on expanding Cantor sets, non-ergodic physical measures) is positive.

Our understanding of the topological properties of the Newhouse set are still quite elusive and a more complete
comprehension of the mechanisms of construction of higher complexity sinks could be the key to shed some more light on
the subject.

We also believe that analogous results can be proved in the conservative setting; building from the renormalization
techniques provided e.g.\ in~\cite{Mora-Romero} and with constructions similar to the ones introduced in~\cite{JMD} and
developed in this paper, we believe possible to show that the Hausdorff dimension of the Newhouse parameter set for a
family of conservative diffeomorphisms which non-degenerately unfolds a homoclinic conservative is bounded below by
$1/4$. 
\appendix
\section{A modern proof of Newhouse Theorem~\ref{Newhouserev}}\label{s_proofNewhouse}
Let us first recall the concept of Newhouse thickness, which will be instrumental in the proof of
Theorem~\ref{Newhouserev}.
\begin{defi} Given a Cantor set $K\subset \R$, a \emph{gap} of $K$ is a connected component of $\R\setminus K$.  Given a
  bounded gap $G$ of $K$ and $u$ in the boundary of $G$, a \emph{bridge} $B$ of $K$ at $u$ is the connected component of
  $u$ in the complement of the union of the gaps longer than $|G|$.  The thickness of $K$ at $u$ is:
  $\tau(K,u)=|B|/|G|$.  The \emph{thickness} of $K$, denoted by $\tau(K)$ is the infimum among these $\tau(K,u)$ for all
  boundary points $u$ of bounded gaps.
\end{defi}
\begin{defi}
  If $K$ is a basic (hyperbolic) set of a $C^2$-diffeomorphism, the \emph{stable thickness} of $K$, denoted by
  $\tau^s(K)$, is defined by
  \[\tau^s(K)= \limsup_{\epsilon\to 0}{\tau(W^u_{\epsilon}(q) \cap W^s_{\eta}(K))},\] where
  $\eta>0$ and $q$ a point in $K$.  Proposition 5 of~\cite{Ne79} states that this definition of thickness is independent
  of $\eta$ and $q\in K$.  Thus it is well defined.  The \emph{unstable thickness} $\tau^u(K)$ is defined in a similar
  manner.
\end{defi}
  The following is the celebrated Newhouse Gap Lemma.
\begin{lemm}[{\cite[Lemma 4]{Ne79}, but see also~\cite[Section 4.2]{PT93}}]\label{GapLemma}
  Let $K_1,K_2\subset \R$ be Cantor sets with thickness $\tau_1$ and $\tau_2$.  If $\tau_1\cdot \tau_2>1$, then one of
  the three following possibilities occurs: $K_1$ is contained in a gap of $K_2$, $K_2$ is contained in a gap of $K_1$,
  $K_1\cap K_2\not= \emptyset$.
\end{lemm}

\subsection{H\'enon like maps close to Chebyshev quadratic polynomial}
We now proceed to construct thick Cantor sets for H\'enon-like maps in a range of parameters which bring them close
to the (one-dimensional) dynamics of the Chebyshev quadratic polynomial.  We proceed in two stages: in the first stage
we construct thick Cantor sets for the one-dimensional dynamics; in the second one, we show that the hyperbolic
continuation of such sets give thick invariant Cantor sets for the two-dimensional dynamics.

Let us first recall that any H\'enon $C^r$-like family:
\[f_a\colon (x,y)\mapsto (x^2+a+y,-b x)+(A_a(x,y),b B(x,y))\] is $C^r$-close to the H\'enon family $(h_{a,0})_a$ defined
by $h_{a,0}(x,y)=(x^2+y+a,0)$.  The restriction of $h_{a,0}$ to the line $\R\times \{0\}$ is equal to $Q_a(x)=x^2+a$.
The map $f_a$ is also $b$-$C^r$-close to the real unimodal map $P(x)=x^2+a+A_a(x,0)$.  Up to conjugacy with a
translation, we can assume that $DP(0)=0$.  Since $P$ is unimodal, it has a two--branched inverse: one branch is orientation
preserving and will be denoted with $P\inv_+$ and the other one is orientation reversing and will be denoted with
$P\inv_-$; observe that, by construction, we have  $P\inv_\pm=(P|{\R^\pm})\inv$.

We will construct the wild basic set for $a$ close to $-2$.  For such $a$, $P$ has two fixed points: denote the
orientation reversing fixed point by $\alpha\approx -1$ and the orientation preserving one by $\beta\approx 2$.  We define
inductively $(\alpha^\pm_{n})_{n\ge 0}$ as follows: let $\alpha^-_{0}=\alpha$ and $\alpha^+_{0}=P\inv_+(\alpha)$; then
we let
\begin{align*}
\alpha_{n}^\pm&:=P\inv_\pm(\alpha^+_{n-1})\in\R^\pm.
\end{align*}
By linearizing the fixed point $\beta$ of $P$, we remark that $(\alpha^+_n)_n$ converges to $\beta$ and
$(\alpha^-_{n})_n$ converges to $P\inv_-(\beta)$.  To fix notations, let $\alpha^+_{\infty}=\beta$ and
$\alpha^-_{\infty}=P\inv_-(\beta)$.  Note that the distance between $\alpha^\pm_{n}$ and $\alpha^\pm_{\infty}$ is of the
same order as $DP(\beta)^{-n}$. 
Define $\tilde\alpha^\pm_1=\alpha^\pm_0$; if the critical value $a$ of $P$ is smaller than $\alpha^-_{n-1}$, we define
$\tilde \alpha^\pm_{n}:=
P\inv_\pm(\alpha^-_{n-1})\in\R^\pm$.  Likewise, if $a<\alpha^-_\infty$, define $\tilde\alpha^\pm_\infty:=P\inv_\pm(\alpha^-_\infty)$.\\
Two Cantor sets will be important for our purposes:
\begin{align*}
  C_1(P)&:= \bigcap_{k\ge 0} P^{-k} \left([\alpha^-_1,\tilde\alpha^-_2]\sqcup[\tilde\alpha^+_2,\alpha^+_1]\right),\text{ if }a< \alpha^-_{1};\\
  C_2(P)&:= \bigcap_{k\ge 0} P^{-k}
  \left([\alpha^-_\infty,\tilde\alpha^-\infty]\sqcup[\tilde\alpha^+_\infty,\alpha^+_\infty]\right),\text{ if
  }a<\alpha^+_{ \infty }.
\end{align*}
For $r\ge 2$, $C_1(P)$ is a hyperbolic Cantor set, and we are going to show that $C_2(P)$ is also a hyperbolic Cantor
set. Then we will compute the thickness of these sets.  The hyperbolic continuation of $C_1(P)$ will be contained in the
wild basic set $\tilde K$ for $f_a$, whereas the hyperbolic continuation of $C_2(\mathcal RP)$ will be embedded in
$\tilde K$ as a basic set $K_2$, where $\mathcal RP$ is a pre-renormalization of $P$. A general picture of the construction is given Figure \ref{f_construction}.

  \begin{figure}[!h]
    \centering
    \includegraphics{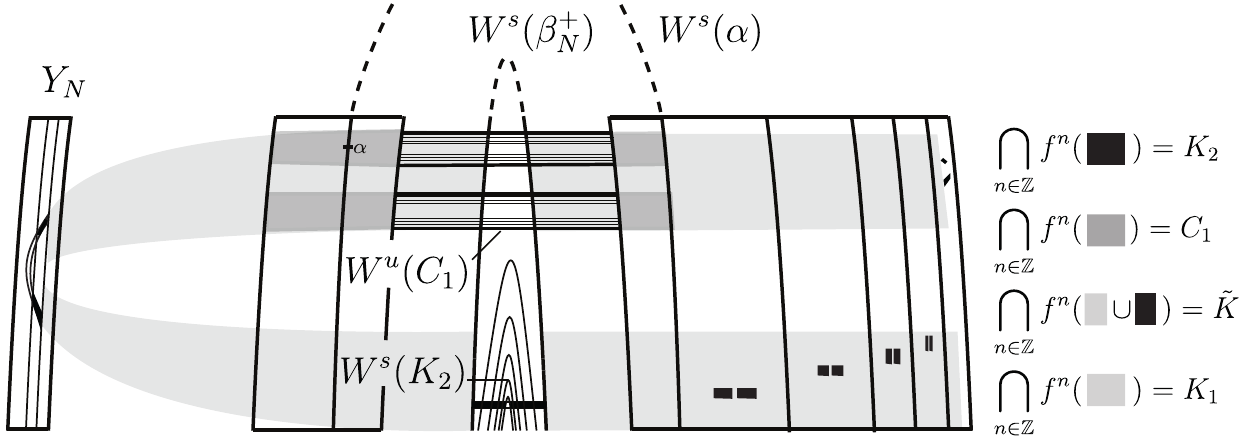}
    \caption{Sketch of the relative positions of the Cantor sets $C_1$, $K_2$ and $\tilde K$.}
    \label{f_construction}
  \end{figure}

Let us now define $I_0=\tilde I_0=[\alpha^-_0,\alpha^+_0]$, and for $n\ge1$:
\begin{align*}
  I^-_n&=[\alpha^-_n,\alpha^-_{n-1}]&
  I^+_n&=[\alpha^+_{n-1},\alpha^+_n]&
  \tilde{I}^-_n&=[\tilde\alpha^-_{n-1},\tilde\alpha^-_n]&
  \tilde{I}^+_n&=[\tilde\alpha^+_n,\tilde\alpha^+_{n-1}]
\end{align*}
when well defined. Then they are diffeomorphically sent onto $I_0$ by $P^n$.
\begin{lemm}\label{Cn}
  There exist $D>0$ and $\epsilon>0$, so that for every $P$ $C^2$-$\epsilon$-close to $Q_{-2}$, whose critical value $a$
  satisfies $a<\alpha^-_\infty$, it holds that $C_2(P)$ is hyperbolic and its thickness is at least
  $D/\sqrt{\alpha^-_\infty-a}$.
\end{lemm}
\begin{proof}
  Let $\mathcal A:= \{\tilde I^{\pm}_n;\; n\ge 1\}\cup \{I_{0}\}$.  Let $\mathcal T$ be formed by the
  pairs $(J ,J')\in\mathcal A^2$ so that $P(J)\supset J'$ and $J\not= I_{0}$.  A sequence of intervals $(J_i)_{i=1}^n\in
  \mathcal A^n$ is \emph{admissible} if $(J_{i},J_{i+1})$ is in $\mathcal T$ for every $i$.
  \begin{sublemm}\label{linearization}
    There exists $C>0$ and $\epsilon>0$, so that for every $P$ $C^2$-$\epsilon$-close to $Q_{-2}$ and satisfying
    $a<\alpha^-_\infty$, for every $N\ge 1$ and every admissible chain $(J_i)_{i=1}^N\in \mathcal A^N$, there exist
    diffeomorphisms $\phi$ and $\psi$ and an affine map $A$ so that:
    \begin{itemize}
    \item the norm of $DA$ is at least $(1.5)^N/C$;
    \item $\phi\circ (P|J_N)\circ \cdots \circ (P|J_2)\circ  (P|J_1)\circ \psi=A$;
    \item the $C^2$-norm of $\phi$, $\psi$ and their inverses are bounded by $C$.
    \end{itemize}
  \end{sublemm}
  This sublemma enables us to prove Lemma~\ref{Cn}: in fact, the central gap $G_0\ni 0$ of $C_2(P)$ has length of the
  order of $\sqrt{\alpha^-_\infty-a}$.  For any other gap $G$ of $C_2(P)$, there exists an admissible chain
  $(J_i)_{i=1}^{N+1}$ ending at $J_{N+1}=I_{0}$, so that $G:= (P|J_1)^{-1}\circ (P|J_2)^{-1}\circ \cdots \circ
  (P|J_N)^{-1}(G_0)$.  An associated bridge $B$ to $G$ has length at least the order of $(P|J_1)^{-1}\circ
  (P|J_2)^{-1}\circ \cdots \circ (P|J_N)^{-1}(B_0)$, where $B_0$ is a connected component of $I_{0}\setminus G_0$.
  Hence, by Sublemma~\ref{linearization}, the thickness associated to the gap $G$ is at least of the order of
  $\frac{|B_0|}{C^2 |G_0|}$, which is of the same order as $1/\sqrt{\alpha^-_\infty-a}$.\end{proof}
\begin{proof}[Proof of Lemma~\ref{linearization}] 
  To prove this linearization, we only need to control the distortion of any admissible chain ending at $I_{0}$.  It is
  well known that such a distortion is bounded if there exists $C>0$ and $\lambda>1.5$, so that for every admissible
  chain $(J_i)_{i=1}^{N+1}$ ending at $J_{N+1}= I_{0}$, the following hyperbolic times inequality holds:
  \[ |DP^N(x)|\ge C \lambda^N,\; \forall x\in (P|J_1)^{-1}\circ (P|J_2)^{-1}\circ \cdots \circ (P|J_N)^{-1}(I_{0}).\]
  Such an inequality is easy to prove when $J_1,\dots, J_N$ is a sequence of the form $I^\pm_{N-1}, I^+_{N-2},\cdots
  ,I^+_1,I_0$.  To prove the general case, it is sufficient to show the existence of a bounded Riemannian metric $g$ on
  $I_0$ so that for every $x\in \tilde I^\pm_n$, we have:
  \[
  |DP^n(x)|>(1.5)^n.
  \]
  The map $Q_{-2}(x)= x^2+a$ is expanding for the metric $g(x)= 1/\sqrt{x^2-4}$, which is bounded on $[-1,1]$.  Hence
  for every $M\ge 0$, if $P$ is close enough to $Q_{-2}$, the restriction $P^n|{\tilde I^{\pm}_n}$ is expanding for the
  metric $g$ and every $M\ge n\ge 2$ ($g|\cup_{n\le M} I_n\cup \tilde I_n$ is bounded).  On the other hand, by using the linearization of the fixed point $\beta$, we
  remark that for $n$ sufficiently large, the derivative of $P^n|{\tilde I^ \pm_n}$ is at least of the same order as
  ${|DP(\beta)|^{n/2}}$.
\end{proof}

Now that we have constructed invariant Cantor sets for the one-dimensional dynamics and that we have obtained bounds on
their thicknesses, we proceed to consider the hyperbolic continuation of such invariant sets and estimate their
thicknesses.
\begin{lemm}\label{thickness2dsimple}
  For any $C^2$-H\'enon like map $f$ close to $h_{-2,0}$, the Cantor set $C_1(Q_{-2})$ persists as a basic set
  $C_1(f)$.  The thickness $\tau^u(C_1(f))$ is of the order of $b^2$, with $b:= |\det\, D_0f|$.
\end{lemm}
Similarly to Lemma~\ref{Cn}, we now prove:
\begin{lemm}\label{Thickness2d}
There exist $C>0$, $\epsilon>0$, so that for every H\'enon-like map
\[f(x,y)=(x^2+y+a,-b x)+(A(x,y),b B(x,y)),\] such that $P(x):=x^2+a+A(x,0)$ is $C^2$-$\epsilon$-close to $Q_{-2}$ and
such that the critical value $a$ of $P$ satisfies $a<\alpha^-_\infty$ and $b\ll(\alpha^-_\infty-a)$,
then the hyperbolic continuation $C_2(f)$ of $C_2(P)$ is such that $\tau^s(C_2(f))\geq C/\sqrt{\alpha^-_\infty-a}$.
\end{lemm}
\begin{proof}%
  The proof is the same as for Lemma~\ref{Cn}, provided that we show that the local unstable manifolds of $C_2(f)$ are
  $C^2$-close to horizontal lines.  We now prove this horizontality condition.

  Let $\beta(f)$ (resp. $\alpha(f)$) be the hyperbolic continuations of the orientation preserving (resp. reversing)
  fixed point $\beta$ (resp. $\alpha$) of $P$.

  First we notice that local stable manifolds of $\beta(f)$ and $\alpha(f)$ are $b$-$C^2$-close to:
  \begin{align*}
    W^s_{\loc}(\beta)&= \{(x,y)\in \R\times [-1,+\infty):\; P(x)+y=\beta\}\\
    W^s_{\loc}(\alpha)&=\{(x,y)\in \R\times [-1,+\infty):\; P(x)+y=\alpha\}
  \end{align*}
	By continuity, the hyperbolic continuation $C_2(f)$ of $C_2(P)$
  must stay below $W^s_{\loc}(\beta)$.  Furthermore, from the analytic expression of $f$, $C_2(f)$ is included in the strip
  $b$-close to:
  \[\{(x,y)\in \R\times [-10 b ,10 b]:\; P(x)+y\le \beta\}\]

  Recall that the central gap $G_0$ of $C_2(P)$ is given by $[\alpha^-_\infty, \alpha^+_\infty]$; similarly,
  hyperbolic continuations $\tilde \alpha^\pm_\infty(f)$ of $\tilde \alpha^\pm_\infty$ have a local stable manifold which
  is $b$-$C^2$-close to:
  \[W^s_{\loc}(\tilde \alpha^\pm_ \infty)= \{(x,y)\in \R\times [-1,+\infty):\; P(x)+y=\alpha^-_\infty\}\] It is a
  parabola with its top at $y$ of the order $\alpha^-_\infty-a$, and so large with respect to $b$.  Hence this parabola
  cuts the strip $\R\times [-10 b ,10 b]$ at a distance of the order of $\sqrt{\alpha^+_{ \infty }-a}$ of the line $x=0$.

  By continuity the set $C_2(f)$ must be above this parabola.  Hence $C_2(f)$ remains at a distance of the order of
  $\sqrt{\alpha^-_\infty-a}$ of the line $x=0$.

  Thus at a neighborhood of $C_2(f)$, the cone \[\chi=\{(u,v)\in \R^2:\; |v|\le |u| \sqrt{\alpha^-_\infty-a} \}\] is sent
  into itself, and so the local unstable manifolds are included in it.  This implies that they are $C^1$-close to
  horizontal curves.

  Let $Y_e$ be the intersection of $\R\times [-10 b, 10 b]$ with the domain below the local stable manifold of
  $\alpha(f)$ and above the local stable manifold of $\tilde \alpha^\pm_\infty(f)$.  We notice that $Y_e$ is close to
  the compact set $([\alpha^-_{0}, \tilde \alpha^-_{\infty }]\cup [\tilde \alpha^+_{\infty },\alpha^+_{0}])\times
  \{0\}$.

  Similarly to the one dimensional case, every point $z\in C_2(f)\cap Y_e $ has infinitely many preimages in $Y_e$.  From
  the same argument as for Lemma~\ref{linearization}, if $z'$ is sent by $f^n$ to $z$, and $w\in \chi$, then:
  \[\|D_z f^n(w)\|\ge \|D_z f^j(w)\|,\quad \forall j\le n.\] By~\cite[Lemma 2.4]{YW}, we can conclude that the
  local unstable of $z$ has a curvature of the same order as $b$.
\end{proof}
\subsection{Misiurewicz pre-renormalization}
In this subsection we quickly sketch the construction of the Misiurewicz pre-renormalization of a smooth family of
H\'enon-like maps; we refer once more the reader to~\cite{MisuRen} for more details.

Let, as before, $(f_a)_a$ be a $C^2$-H\'enon like family of class $C^4$.  Hence $f_a$ has the form
$f_a(x,y)= (x^2+a+y,-bx )+(A_a(x,y),bB_a(x,y))$.  Put $P_a:= x^2+a+A_a(x,0)$.

The set $K(P_a):= \{ \alpha^\pm_n(P_a); \; n\geq 0\}\cup \{\alpha^\pm_\infty(P_a)\}$ is hyperbolic for $P_a$.  Let
$K(f)$ be a hyperbolic continuation of $K(P_a)\times \{0\}$ for a map $f$.  Hence for every continuation $k\in K(f)$
of $(x_0,0)\in K(P_a)\times \{0\}$, a local stable manifold $W^s_{\loc}(k,f)$ of $k$ is $C^2$-close to the connected
component of $(x_0,0)$ in $\{(x,y)\in \R\times [-0.2,0.2]:\; x^2-2+y=x_0^2-2\}$.

For any $n\ge 0$, let $\hat Y^{\pm}_{n,f}$ be the rectangle bounded by the local stable manifolds $W^s(k,f)$ associated to
the boundary points $k\in \partial I^\pm_{n}$, and the lines $\{(x,y)\in\R^2:y=\pm 0.1\}$; let furthermore $\hat I^\pm_{ n,f }:= \hat Y^\pm_{k,f }\cap \{(x,y)\in \R^2: y=0\}$.

Observe that by linearizing the fixed point $\beta^+$, there exists $C>0$, such that for every $0\le k\le n$, we have:
\[
\forall x\in \hat I^\pm_{n,f},\; |DP_{-2}^k(x)|\ge C |DP_{-2}(\beta^+)|^k.\] By~\cite[Section 2.2.2]{MisuRen}, there exists $\epsilon
>0$, a $C^4$-neighborhood $V$ of the H\'non map $h_{-2,0}$, so that for each $f\in V$ which is $C^2$-H\'enon like, any $n\ge 0$, there
exist $C^2$-coordinates $y_{n,f}$ on an $\epsilon$-neighborhood $Y^-_{n,f}$ of $\hat I^-_{n,f}$ in $\hat Y^-_{n,f}$ and
$y'_{n,f}$ on an $\epsilon$-neighborhood $Y_{0,f}$ of $\hat I_{0,f}$ in $\hat Y_{0,f}$:
\begin{align*}
  y_{n,f}&\colon \hat I^-_{n,f}\times [-\epsilon,\epsilon]\to Y^-_{n,f}&%
  y'_{n,f}&\colon \hat I_{0,f}\times [-\epsilon,\epsilon]\to Y_{0,f};
\end{align*}

so that there exist $\sigma_{n,f}>3^{n}$ and $\lambda_{n, f}$ satisfying for $n\ge 6$ the following properties
\begin{itemize}
\item $y'^{-1}_{n,f} \circ f^n\circ y_{n,f}(x,y)= (\sigma_{n,f} \cdot x, \lambda_{n,f} \cdot y)$, for $(x,y)\in \hat I^+_{n,f}\times
  [-\epsilon ,\epsilon]$;
\item derivatives up to the second order of $(f,x,y)\mapsto y'_{n,f}(x,y)$, $(f,x,y)\mapsto y_{n,f}(x,y)$ and their
  inverses $(f,x,y)\mapsto y'^{-1}_{n,f}(x,y)$, $(f,x,y)\mapsto y_{n,f}^{-1}(x,y)$ are bounded independently of $f\in V$
  which is $C^2$-H\'enon like (even if the determinant of $Df$ is zero!).
\end{itemize}

Since the unfolding of the homoclinic tangency is nondegenerate, for every $C^2$-H\'enon like family $(f_a)_a$, for every
$n\ge 2$, there exists $a_n$, with $a_n\approx-2$ when $n$ is large, so that for $f_a\in V$, $y_{n,f_a}^{-1} \circ
f_a\circ y'_{n,f_a}(x,y)$ is of the form:
\[
y_{n,f_a}^{-1} \circ f_a\circ y'_{n,f_a}(x,y) \colon (p+x, y)
\mapsto (\xi x^2+\theta (a-a_n) +\gamma \cdot y, q+\zeta \cdot x)+ E_a (x, y)\in \R^2,
\]
where $p\in \text{int}\, \hat I_{0} $, $q\in (-\epsilon,\epsilon)$, $\zeta$ are constants, and $\xi$, $\theta$ and $\gamma$ are
non-zero constants (independent of $a$)\footnote{These constants $p$, $q$, $\zeta$, $\xi$, $\theta$ and $\gamma$ depend
  on $n$ and depend smoothly on the H\'enon like family $(f_a)_a$.} and $E_a =(E_{a}^1,E_{a}^2)\in C^{2}(\R\times \R^2,
\R^2)$ satisfies at $P:=(p,0)$:
\begin{equation*}
  \left\{\begin{array}{l}
    E^1_{a_n}(P)=\partial_x E_{a_n}^{1}(P)=\partial_a E_{a_n}^{1}(P)=\partial_y
  E_{a_n}^{1}(P)=\partial_{xx} E_{a_n}^{1}(P)=0\\ 
  E_{a_n}^2(P)=\partial_x E_{a_n}^2(P)=0.
\end{array}\right.
\end{equation*}

By~\cite[Theorem 3.1]{MisuRen}, the family $( y'^{-1}_{n,f_a} \circ f_a^{n+1}\circ y'_{n,f_a})_a$ is conjugated to the
H\'enon $C^2$-like map $\mathcal R_nf_{a^{(n)}}$, with $\mathcal R_n f_{a^{(n)}(a)}= \Psi_{f_a}\circ
y'^{-1}_{n,f_a}\circ f_a ^{n+1}\circ y'_{n,f_a} \circ \Psi_{f_a}^{-1}$, where
\begin{align*}
  \Psi_{f_a}(x,y)&:= \xi\cdot \sigma_{n, f_a}\cdot (x- p ,  \sigma_{n, f_a} {\gamma y}-\sigma_{n, f_a} \lambda_{n, f_a} q),\\
  a^{(n)}(a)&:= \sigma_{n, f_a}^2 \cdot(\xi \theta(a-a_n) +\lambda_{n, f_a} \xi \gamma q-\frac{\xi p}{\sigma_{n,
      f_a}}).
\end{align*}

 Moreover $\mathcal R_n f_{a^{(n)}(a)}$ is H\'enon $C^2$-$\delta$-like with $\delta$ small when $n$ is large.

 \begin{rema}\label{rkdet}%
   If the determinant $b$ of $D f_a(0)$ is small, the determinant of  $D\mathcal R_n f_{a^{(n)}(a)}$ is dominated by $b^{n+1}$ when $n$ is large.
 \end{rema}
\subsection{Our proof of Theorem~\ref{Newhouserev}}
We have now introduced all the results which we need in order to give the
\begin{proof}[Proof of Theorem~\ref{Newhouserev}]
  Let us consider a H\'enon $C^4$-like family $(f_a)_a$: recall that $f_a$ is of the form
  \[f_a\colon (x,y)\mapsto (x^2+a+y,-b x)+(A_a(x,y), b B_a(x,y)).\] Define $P_a(x):= x^2+a+A_a(x,0)$ and $\breve
  f_a(x,y):=(P_a(x)+y,0)$.  For $a<-1$, let $\beta^+_a$ be the orientation preserving fixed point of $P_a$ and
  $\beta^-_a:=(P_a|\R^-)^{-1}(\beta^+_a)$.  We consider $a\approx-2$ such that $\beta^-_a-a$ is small and positive.  For
  such values of $a$ we can define the Cantor set $C_1(P_a)$; the stable thickness of its hyperbolic continuation
  $C_1(f_a)$ is of order $b^2$ by Lemma~\ref{thickness2dsimple}.

  Also for many such values of $a$ we can perform the Misiurewicz pre-renormalization $\mathcal R_n\breve f_{a^{(n)}}$
  and $\mathcal R_n f_{a^{(n)}}$ which we have defined in the previous section, for large $n\ge 2$.  Actually $\mathcal
  R_n\breve f_{a^{(n)}}$ preserves the horizontal line $\R\times \{0\}$ and its restriction is a unimodal map $\mathcal
  R_n P_{a^{(n)}}$.  Let $\epsilon$ be the parameter appearing in Lemmata~\ref{Cn} and~\ref{Thickness2d}; then
  there exists $N_\epsilon$ so that for $N=N_\epsilon$ large enough and $a^{(N)}=a^{(N)}(a)\in [-3,0]$, this
  renormalization is a map $C^2$-$\epsilon$-close to the quadratic map $x^2+a^{(N)}$.  Note that the value of $N_\epsilon$
  depends only on $\epsilon$ and the $C^4$-norm of $(f_a)_a$, as far as the family is $C^2$-H\'enon like.  We assume
  $a^{(N)}$ close to $-2$.  More precisely, if $\beta'^+_{a^{(N)}}$ denotes the orientation preserving fixed point of
  $\mathcal R_N P_{a^{(N)}}$ and $\beta'^-_{a^{(N)}}:=(\mathcal R_N P_{a^{(N)}}|\R^-)^{-1}(\beta'^+_{a^{(N)}})$, we assume
  $b$ and $a$ so that:
  \begin{equation}
    \tag{$\mathcal C_1$}  0<\beta'^-_{a^{(N)}}-a^{(N)}(a)\ll  b^4.\end{equation}
  Hence by Lemma~\ref{Cn}, the Cantor set $C_2(\mathcal R_N P_{a^{(N)}})$ has thickness large with respect to $b^{-2}$.

  Also the same pre-renormalization of $(f_a)_a$ defines a map $(\mathcal R_N f_{a})_a$ which is close to $(\mathcal R_N
  \breve f_{a})_a$.  By Remark~\ref{rkdet}, the determinant $b_N$ of $D_0\mathcal R_N f_{a}$ is of the order of
  $b^{N+1}$, and so $\mathcal R_N f_{a}$ is $C^2$-$b^{N+1}$-close to $\mathcal R_N \breve f_{a}$ (for a large $N$).  In order to apply
  Lemma~\ref{Thickness2d}, we assume $b$ and $a$ so that at $a^{(N)}=a^{(N)}(a)$:
  \begin{equation}
    \tag{$\mathcal C_2$} 0<b^{n+1}\sim b_N\ll  \left(\beta'^-_{a^{(N)}}-a^{(N)}\right)^2.
  \end{equation}
  Then, by Lemma~\ref{Thickness2d}, there exist a constant $C$ and a hyperbolic horseshoe $C_2(\mathcal R_N f_{a^{(N)}})$ of
  stable thickness greater than $C/\sqrt{\beta'^-_{a^{(N)}}-a^{(N)}}\gg  C/b^{2}$.

  We remark that conditions $(\mathcal C_1)$ and $(\mathcal C_2)$ are simultaneously possible as far as $b^{8}$ is
  large with respect to $b^{N+1}$ (e.g. we shall assume $N\ge 8$).  Moreover, for $N$ large but bounded, condition $(\mathcal C_2)$ remains possible
  for $b$ arbitrarily small.

  The hyperbolic horseshoe $C_2(\mathcal R_N f_{a^{(N)}})$ is embedded into the following basic set of $f_a$:

  \[K_2(f_a):= \bigcup_{k=0}^{N-1} f^k_a(y'_{N,a} \circ \Psi_{f_a}^{-1}( C_2(\mathcal R_N f_{a^{(N)}}) )),\]
  where $y'_{N,a}$ and $\Psi_{f_a}^{-1}$ are the charts defining the Misiurewicz renormalization.

  Since $C_2(\mathcal R_Nf_{a^{(N)}})$ is embedded into $K_2(f_a)$, we conclude that $K_2(f_a)$ has also stable thickness
  large with respect to $b^2$.  Consequently,
  \begin{equation}
    \tau^u(C_1(f_a))\cdot \tau^s(K_{2}(f_a)\gg 1.\label{e_thickHorseshoes} 
  \end{equation}
  We notice that, for some Riemannian metric, the expansion of $C_1(f_a)$ is close to $2$, whereas the expansion of $K_2(f_a)$
  is close to $\sqrt[N]{2}$.  This observation will be useful to prove condition $(i)$ of Theorem~\ref{Newhouserev}.

  Note that if we prove that there exist a basic set $\tilde K$ which contains both $C_1(f_a)$ and $K_{2}(f_a)$ and
  which has a homoclinic tangency, then by~\eqref{e_thickHorseshoes} and the Gap Lemma~\ref{GapLemma} it is a wild
  hyperbolic set; this implies item (ii) of our statement.  Let us first take care of constructing $\tilde K$: the
  reader might find useful to refer to Figure~\ref{f_construction} while reading through the description which follows.

  We will use the following result, which is indeed
  just a rephrasing of~\cite[Lemma 8]{Ne79}:
  \begin{lemm}\label{l_lemma8Newhouse}
    Assume there exist $z_1, z'_1\in C_1(f_a)$ and $z_2,z_2'\in K_{2}(f_a)$, local unstable manifolds $W^u_{\loc} (z_1)$,
    $W^u_{\loc} (z_2)$ and local stable manifolds $W^s_{\loc} (z'_1)$, $W^s_{\loc} (z'_2)$ such that $W^u_{\loc} (z_1)$
    intersects $W^s_{\loc} (z'_2)$ transversely at a point $w$, $W^u_{\loc} (z_2)$ intersects $W^s_{\loc} (z'_1)$
    transversely at a point $w'$, and $w,w'\notin C_1(f_a)\cup K_2(f_a)$.  Then there exists a basic set $\tilde K_a$
    containing both $C_1(f_a)$ and $K_{2}(f_a)$.
  \end{lemm}
  Moreover, in order to prove condition $(i)$ of Theorem~\ref{Newhouserev}, we shall also prove that the intersection
  points $w,w'$ enjoy an expansion which is uniformly bounded from below when $b$ is small.

  To this end, let us go back to the one dimensional dynamics, which appears at the limit when $b\to0$.  More
  precisely, when $b$ approaches $0$, Conditions $(\mathcal C_1)$ and $(\mathcal C_2)$ imply $a$ to be such that $a\in
  I_N$, and $P^{N+1}_a|P^{-1}_a(I_N)$ is conjugated to the Chebyshev map (both restricted to their maximal invariant
  compact set).

  Let $\beta^+_N$ be the orientation preserving fixed point of $P_a^{N+1}|P^{-1}(I_N)$.  Let $\beta^-_N$ be the preimage
  of $\beta^+_N$ by $P_a^{N+1}|P_a^{-1}(I_N)$.  From classical unimodal theory, the maximal invariant set:
  \[
  K_1(P_a):= \bigcap_{n\ge 0} P^{-n}_a ([\beta^-_a,\beta^+_a]\setminus (\beta^+_N, \beta^-_N))
  \]
  is hyperbolic (one can even show that its expansion is at least of the order of $\sqrt[N]{2}$) and transitive.  This
  basic set persists as a basic set $K_1(f_a)$ for $b$ small, with a similar expansion.

  We notice that $K_1(f_a)\cap K_2(f_a)$ contains the hyperbolic continuation $\beta_N^+(f_a)$ of $\beta^+_N$.  Since
  both $K_1(f_a)$ and $K_2(f_a)$ are hyperbolic Cantor sets, we can find $z_1, z'_1\in K_1(f_a)$ and $z_2,z_2'\in
  K_{2}(f_a)$ all close to $\beta_N^+(f_a)$ such that $W^u_{\loc} (z_1)$ intersects $W^s_{\loc} (z'_2)$ transversely at a
  point $w$, $W^u_{\loc} (z_2)$ intersects $W^s_{\loc} (z'_1)$ transversely at a point $w'$, and $w,w'\notin K_1(f_a)\cup
  K_2(f_a)$.  Hence, using Lemma~\ref{l_lemma8Newhouse} we prove the existence of a basic set $\tilde K_a$ containing
  both $K_1(f_a)\supset C_1(f_a)$ and $K_2(f_a)$, with expansion bounded from below uniformly on $b$ (actually at least
  of the order of $\sqrt[N]{2}$).

  In order to conclude the proof, we now need to show that the basic set $\tilde K_a$ has a homoclinic tangency for some
  $a$ which satisfies $(\mathcal C_1)$ and $(\mathcal C_2)$; to this end observe that the local stable manifold
  $W^s_{\loc}(\beta_N^+(f_a))$ is $b$-$C^2$-close to the parabola
  \[
  \{(x,y)\in [-1,+\infty):\; x^2+y=\beta_N^+\}.
  \]
  Since $K_1(f_a)$ contains $\beta_N^+(f_a)$, the local stable manifold $W^s_{\loc}(\beta_N^+(f_a))$ is $C^2$-accumulated by
  other local stable manifolds $(W^s_n(a))_n$ of points in $K_1(f_a)$; the parabolas $W^s_n$ are above
  $W^s_{\loc}(\beta_N^+(f_a))=:W^s_\infty(a)$.


  If $\sigma$ denotes the unstable expansion of $Df^{N+1}(\beta_N^+(f_a))$, there exists $C>0$ independent of $(f_a)_a$
  (and so of $b$), but depending on $N$ already fixed, such that we can assume, for any $n$:
  \[\sigma^{-n}/C < d( W^s_n(a) , W^s_\infty(a))< C\sigma^{-n}.\]
  By taking $C$ larger, we can examine the images of $( W^s_n )_{n\in \{0,\dots,\infty\}}$ by the renormalization chart
  $\phi_a:= \Psi_{f_a}\circ y_{N,f_a}'^{-1}$ and obtain, for any $n$:
  \[\sigma^{-n}/C < d(\phi_a( W^s_n) , \phi_a(W^s_\infty(a)))< C\sigma^{-n}\] 
  Furthermore, the parameter dependence of $(\phi_a( W^s_n))_{n\in \{0,\dots,\infty\}}$ depends continuously on $n$ in
  the one point compactification $\{0,\dots,\infty\}$ of $\mathbb N$.


  On the other hand, a certain local unstable manifold $W^u$ of $\phi_a(K_2(f_a))$ is $b^{N+1}$-$C^2$-close to the curve
  \[
  \{\mathcal R_N f_{a^{(N)}(a)} (t,0):\; t\in[-1,1]\}.
  \]
  This curve has a fold which is $b^{N+1}$-close to $(a^{(N)}(a),0)$.  Also, there exists $a^{(N)}_\infty$ which is 
  $b^{N+1}$-close to $\beta'^{-}_{a^{(N)}_\infty}$ so that $W^u(a^{(N)}_\infty)$ is tangent to the local stable manifold
  $\phi_a( W^s_\infty(a^{(N)}_\infty))$, and the unfolding of this tangency is nondegenerate. So there are parameters
  $a_n$ 
  so that $W^u(a^{(N)}(a_n))$ is tangent to the local stable manifold $\phi_a( W^s_\infty(a^{(N)}(a_n)))$ and
  satisfy (by taking $C$ larger) for any $n$:
  \[\sigma^{-n}/C < d(a^{(N)}_n , a^{(N)}_\infty)< C\sigma^{-n}.\]
  Consequently there exists $n$ so that $a^{(N)}(a_n)$ satisfies $(\mathcal C_1)$ and $(\mathcal C_2)$.
\end{proof}

\bibliographystyle{alpha}
\bibliography{references}

\end{document}